\documentclass[12pt,a4paper]{amsart}
\usepackage[margin=1in]{geometry}
\usepackage{enumitem}
\usepackage{amsmath}
\usepackage{physics}
\usepackage{amssymb}
\numberwithin{equation}{section}
\usepackage{graphicx}
\usepackage{amssymb}
\usepackage{amsthm}
\usepackage{enumitem}
\usepackage{tikz}
\usetikzlibrary{shapes}
\usepackage{mathrsfs}
\usepackage{hyperref}
\usepackage[capitalise]{cleveref}
\usetikzlibrary{arrows}
\usepackage[utf8]{inputenc}
\usetikzlibrary{calc}
\numberwithin{equation}{section}
\usepackage{mathptmx}
\usepackage{tikz-cd}
\usepackage[T1]{fontenc}

\usepackage{tikz-cd}
\usepackage{lipsum} 
\let\oldpart\part
\renewcommand\part{\newpage\oldpart}

\usepackage[%
]{amsrefs}
\hypersetup{%
	linktoc=all,
	colorlinks=true,
	linkcolor=brown,
	citecolor=cyan,
}
\newcommand{\normord}[1]{\xcentcolon\mathrel{#1}\xcentcolon}
\newcommand{\xcentcolon}{%
  \mathrel{\vbox{\hbox{$:$}\kern.2ex}}%
}

\newtheorem{theorem}{Theorem}[section]
\newtheorem*{theorem*}{Theorem}

\newtheorem{corollary}[theorem]{Corollary}
\newtheorem{lemma}[theorem]{Lemma}

\theoremstyle{definition}
\newtheorem{definition}[theorem]{Definition}

\allowdisplaybreaks

\newcommand{\N}{\mathbb{N}}
\newcommand{\Z}{\mathbb{Z}}

\newcommand{\R}{\mathbb{R}}

\newtheorem{innercustomgeneric}{\customgenericname}
\providecommand{\customgenericname}{}
\newcommand{\newcustomtheorem}[2]{%
	\newenvironment{#1}[1]
	{%
		\renewcommand\customgenericname{#2}%
		\renewcommand\theinnercustomgeneric{##1}%
		\innercustomgeneric
	}
	{\endinnercustomgeneric}
}
\newcustomtheorem{customthm}{Theorem}

\def \ssq{\subseteq}
\def \vac {\mathbf{1}}
\def\sl{\sl}

\def \End {\mathrm{End}}

\def \wt {\mathrm{wt}}

\def\L{\mathcal{L}}

\def\spn{\mathrm{span}}

\def\sl{\mathfrak{sl}_2}
\def\hatsl{\widehat{\sl}}

\def\<{\langle}
\def\>{\rangle}

\def\L{\mathsf{L}}
\def\R{\mathsf{R}}
\def\NN{\mathsf{N}}

\def\U{\scrU}

\def\L{\mathsf{L}}
\def\sR{\mathsf{R}}
\def\U{\mathsf{U}}
\def\A{\mathsf{A}}

\def\scrU{\mathscr{U}}

\begin{document}
	\title{On the strong unital property for the affine VOAs}
	
	\author{Angela Cai}
	\address{Department of Mathematics, University of Pennsylvania, Philadelphia, PA, 19104}
	\email{caian@sas.upenn.edu}

\begin{abstract}
Representations of vertex operator algebras $V$ (VOAs) have numerous applications, including the construction of sheaves of conformal blocks on moduli spaces of curves. For a $V$-module $W = \oplus W_d$, a sequence of associative algebras $\mathfrak{A}_d$ acts on each graded component $W_d$. When these $d${\tiny{th}}-mode transition algebras $\mathfrak{A}_d$ are strongly unital - meaning they are unital with units acting as the identity on $W_d$ - the associated sheaves of conformal blocks are vector bundles rather than merely coherent sheaves. This strong unital property, while difficult to verify in practice, has other important implications as well. In this paper, we construct explicit strong units for $L_{\hatsl}(1,0)$,  the simple affine VOA for $\mathfrak{sl}_2$ at level $1$, and establish that mode transition algebras for universal affine VOAs for $\mathfrak{sl}_2$ are never strongly unital at any level $k$ not equal to the critical level $-2$. 
\end{abstract}
	
	\maketitle
	
	
	\section{Introduction}
    Representations of a vertex operator algebra $V$ (a VOA) are important in mathematical physics and in an array of areas of mathematics. For instance, admissible $V$-modules give rise to sheaves of coinvariants and conformal blocks on moduli of curves. The mode transition algebra $\mathfrak{A}$ is an invariant of a VOA consisting of a direct sum of associative algebras $\mathfrak{A}_d$, $d\in \mathbb{N}$, constructed as quotients of the universal enveloping algebra of a VOA \cites{DGK24, DGK25}. 
The $d$-th mode transition algebras $\mathfrak{A}_d$ act on the degree $d$ component $W_d$ of any admissible  $V$-module  $W=\oplus W_d$.  If  $\mathfrak{A}_d$ admits a unit $\mathscr{I}_d$ and 
$\mathscr{I}_d$ acts as the identity on $W_d$ for all admissible $V$-modules $W$, then one says that $\mathfrak{A}$ is strongly unital (see also \cref{eq:defstrongunit} and \cref{strongUnitThm}). 

There are a number of applications if $\mathfrak{A}$ is strongly unital.  For instance, coherent sheaves of coinvariants and conformal blocks defined by $V$-modules satisfy a property called smoothing if and only if $\mathfrak{A}$ is strongly unital.  If they are coherent and satisfy smoothing, then these sheaves form vector bundles \cite{DGK24}.  Moreover, $\mathfrak{A}$ is strongly unital if and only if the category of admissible $V$-modules is equivalent to the category of modules over Zhu's algebra \cite{GL26}.

The mode transition algebras for rational VOAs are known to have this strong unital property, although the actual units have not been given by a concrete formula. One also knows that strong units exist for some non-rational VOAs, via direct constructions of their units \cite{DGK24}.  Examples, even in the rational case, could lead to a better understanding of which mode transition algebras may have this strong unital property in the non-rational case.

Towards this end, here in \cref{strongUnitThm} we give an explicit formula for the strong units for the $d$-th mode transition algebras for the simple affine VOA for $\mathfrak{sl}_2$ at level $1$. This is arguably the most
basic example of a rational VOA \cites{FK80,FZ92}. One reason for the lack of examples mentioned above is that constructing strong units requires knowing all the relations of the universal enveloping algebra of the VOA, which we determine in this paper for $L_{\hatsl}(1,0)$.

Mode transition algebras of the non-rational Heisenberg and bosonic ghost VOAs are known to have strong units \cites{AB23a, AB23b, DGK24, BFOHY25}. These results were proved by construction. It is therefore natural to ask whether the strong unital property holds for mode transition algebras of $V_{\hatsl}(k,0)$, the universal affine or the vacuum module VOAs at level $k \in \mathbb{C}$, which  are not rational. Here in \cref{noStrongUnitThm}, we show that if the mode transition algebras for these VOAs admit units, then such units will never have the strong unital property  for any $k \in \mathbb{C}$ with $k$ different from the critical level $-2$.

To give more details about these results and to give an outline of the proofs, we now introduce some notation. 

Let $\scrU=\scrU(V)$ be the (completed) universal enveloping algebra associated to a VOA $V$ (see \cref{section2.2} and \cref{UEA}). By definition, $\scrU$ is a topological associative graded algebra with left (resp. right) neighborhood basis $\NN_\L^n \scrU=\scrU\cdot \scrU_{\leq -n-1}$ (resp. $\NN^n_\R \scrU=\scrU_{\geq n+1}\cdot \scrU$) for $n\in \N$. The generalized Verma module functor \cites{DLM1,DGK25}, $\Phi^\L: \mathsf{Mod}(\A)\rightarrow \mathsf{Adm}(V)$ is defined by $ S\mapsto \Phi^\L(S)=(\scrU/\NN^1_\L\scrU)\otimes_{\scrU_0} S$ and $\Phi^\R: \mathsf{Mod}(\A)\rightarrow \mathsf{Adm}(V)$ is defined by $ S\mapsto \Phi^\R(S)=S \otimes_{\scrU_0} (\scrU/\NN^1_\R\scrU)$. The mode transition algebra $\mathfrak{A}$ is the following tensor product:
\[ \mathfrak{A}=\Phi^\R(\Phi^\L(\A))=(\scrU/\NN^1_\L\scrU)\otimes_{\scrU_0} \A\otimes_{\scrU_0} (\scrU/\NN^1_\R\scrU).\] Here $\A$ is the Zhu algebra of $V$ \cite{Z}, which is an important associative algebra attached to a VOA. Since $\scrU/\NN^1_\L\scrU$ (resp. $\scrU/\NN^1_\R\scrU$) is $\Z_{\geq 0}$ (resp. $\Z_{\leq 0}$) graded, the mode transition algebra $\mathfrak{A}$ is bi-graded $\mathfrak{A}=\bigoplus_{m,n\in \N} \mathfrak{A}_{m,-n}$, where $\mathfrak{A}_{m,-n}=(\scrU/\NN^1_\L\scrU)_m\otimes_{\scrU_0} \A\otimes_{\scrU_0} (\scrU/\NN^1_\R\scrU)_{-n}$. The bi-graded piece 
\begin{equation}
\mathfrak{A}_d=\mathfrak{A}_{d,-d}=\scrU_d/\NN^1_\L\scrU_d\otimes_{\scrU_0} \A\otimes_{\scrU_0} \scrU_{-d}/\NN^1_\R\scrU_{-d}
\end{equation} 
is called the $d$-th mode transition algebra. The associative product
\[
\star: \mathfrak{A} \times \mathfrak{A} \to \mathfrak{A}, \quad (\alpha_i \otimes x \otimes \beta_{-j}, \alpha_k \otimes y \otimes \beta_{-l}) \mapsto \alpha_i \otimes x(\beta_{-j} \circledast \alpha_k)y \otimes \beta_{-l} 
\] where $$\circledast: \scrU_{-d}/\NN^1_\R\scrU_{-d}\times \scrU_d/\NN^1_\L\scrU_d \rightarrow \A=\scrU_0/\NN^1_\L\scrU_0,\quad (\alpha,\beta)\mapsto [\alpha\beta]_0$$
 makes $\mathfrak{A}_{d,0}=(\scrU/\NN^1_\L\scrU)_d\otimes_{\scrU_0} \A$ (resp. $\mathfrak{A}_{0,-d}=\A\otimes_{\scrU_0} (\scrU/\NN^1_\R\scrU)_{-d}$) a left (resp. right) module over $\mathfrak{A}_d$. An element $\mathscr{I}_d\in \mathfrak{A}_d$ is called a strong unit if it satisfies 
\begin{equation}\label{eq:defstrongunit}
\mathscr{I}_d\star \alpha=\alpha,\quad \beta\star \mathscr{I}_d=\beta,\quad \mathrm{for}\quad \alpha\in \mathfrak{A}_{d,0},\ \beta\in \mathfrak{A}_{0,-d}. 
\end{equation}

The first part of this paper proves that the mode transition algebras of $V_{\hatsl}(k,0)$ do not admit strong units at any level $k \neq -2$. To prove this theorem, we give a simpler description of the mode transition algebras by utilizing the following isomorphism of graded topological algebras \cite{FZ92}*{Theorem 2.4.3}: 
\begin{equation}\label{isomPhi}
\scrU(V_{\hat{\mathfrak{g}}}(k,0))\overset{\phi}\cong \widetilde{U}(\hat{\mathfrak{g}}, k),\end{equation}
where $\widetilde{U}(\hat{\mathfrak{g}}, k)$ is the quotient of a completion of the universal enveloping algebra of the affinization $\hat{\mathfrak{g}}$ (see \cref{sec3.2}). From now on, denote a homogeneous element $a \otimes t^n \in U(\hat{\mathfrak{g}})$ as $a(n)$.

 The conditions for admitting a strong unit impose strong constraints. Namely, we show that if the strong unit 
$\mathscr{I}_1\in \mathfrak{A}_1$ exists for $V_{\hatsl}(k,0)$, then it must take the form $\mathscr{I}_1 = \sum_{\alpha, \beta} \lambda_{\alpha, \beta} \alpha(-1) \otimes x_{\alpha, \beta} \otimes \beta(1)$, where $\alpha,\beta \in \{e,h,f\}$. We prove that the set $\{e(1) \otimes 1, f(1) \otimes 1, h(1) \otimes 1\}$ is a linearly independent $\A$-basis for $\mathfrak{A}_{0, -1}$ (see \cref{linearInd}). We then use a system of equations from the strong unital property to derive a contradiction with the coefficients $\lambda_{\alpha,\beta}$. More specifically, we show the following:

\begin{customthm}{A}[\cref{noStrongUnitThm}]\label[theorem]{mainA}
    The vacuum module VOA $V_{\hatsl}(k,0)$ does not satisfy the strong unital property for all non-critical level $k \in \mathbb{C}$ with $k \neq -2$.
\end{customthm}

The simple affine VOA $L_{\widehat{\mathfrak{sl}}_2}(k,0)$ is a quotient of the universal affine VOA $V_{\widehat{\mathfrak{sl}}_2}(k,0)$ by the unique maximal submodule.  This quotient is rational, and hence it is known that the mode transition algebras satisfy the strong unital property \cite{DGK24}.  In other words, the constraints which lead to a contradiction for the universal affine VOA are resolved when taking a quotient.

We devote the second part of the paper to constructing the strong units for the $d$th-mode transition algebras of $L_{\hatsl}(1,0)$. We must first determine the spanning set and relations for the universal enveloping algebra of $L_{\hatsl}(1,0)$. While these formulas are expected, they have not appeared in the literature. To prove this result, we give a simpler description of the mode transition algebras by proving the following isomorphism of graded topological algebras:


\begin{customthm}{B}[\cref{Uofsl2}]\label[theorem]{mainB}
There exists a continuous isomorphism between topological associative algebras
\[
\tilde{\psi}:
\scrU(L_{\widehat{\mathfrak{sl}}_2}(1,0))
\overset{\cong}\to 
\widetilde{U}(\widehat{\mathfrak{sl}}_2, 1) / \overline{\langle e(-1)e(-1) \rangle},
\]
where $ \overline{\langle e(-1)e(-1)\rangle}$ is the closure of the two-sided ideal generated by $e(-1)e(-1)\in \widetilde{U}(\widehat{\mathfrak{sl}}_2, 1)$. 
\end{customthm}
 To prove this theorem, we first show that $e(m)e(n)=0$ for all $m,n\in \mathbb{Z}$ in the quotient algebra $\widetilde{U}(\widehat{\mathfrak{sl}}_2, 1) / \overline{\langle e(-1)e(-1)\rangle}$ (see \cref{baseCase,allZero}). These relations are necessary to show that the isomorphism from  $\cref{isomPhi}$ descends to an isomorphism $\tilde{\psi}$ between their respective quotients.

As a Corollary of \cref{mainA}, we have a dense set of spanning elements in $\scrU(L_{\widehat{\mathfrak{sl}}_2}(1,0))$. The elements of $\scrU(L_{\widehat{\mathfrak{sl}}_2}(1,0))$ are spanned by monomials of length at most two by applying the appropriate Lie bracket action to the ideal generated by $e(-1)e(-1)$.

\begin{customthm}{C}[\cref{relationsforUsl2}]\label[theorem]{main:C}
   The following subset is dense in the topological associative algebra $\scrU(L_{\widehat{\mathfrak{sl}}_2}(1,0))$, where we use the same notation for elements in the quotient: 
\[
\operatorname{span}\bigl\{1,\  e(m),\ f(n),\ h(r),\ h(n_1)h(n_2) \,\bigm|\, m,n,r,n_1,n_2 \in \mathbb{Z}\}.
\] 
Moreover, we have the following product relations among  $e(m), f(n)$, and $h(r)$: 
\begin{equation}\label{eq:rel}
    \begin{aligned}
& e(x)e(y)=f(x)f(y)=0, && e(x)h(y) = - e(x+y),\\
&h(x)f(y) = -f(x+y), && h(x)h(y) + h(x+y) = 2e(r)f(x+y-r),
    \end{aligned}
\end{equation}
for all $x,y,r\in \Z$. 
\end{customthm}

 Let  $\scrU=\scrU(L_{\widehat{\mathfrak{sl}}_2}(1,0))$. When we pass to the Zhu algebra $\A(L_{\widehat{\mathfrak{sl}}_2}(1,0))=\scrU_0/\NN^1_\L\scrU_0$ \cites{FZ92,H17}, the relations from \cref{eq:rel} reduce to the following relations: 
\[e^2=f^2=0,\quad eh=-e,\quad hf=-f,\quad h^2+h=2ef. \] allowing us to denote $a(0)$ by $a$ for $a\in \mathfrak{sl}_2$. These are exactly the defining relations in $\A(L_{\widehat{\mathfrak{sl}}_2}(1,0))\cong U(\mathfrak{sl}_2)/\langle e^2\rangle$. Therefore, \cref{main:C} can be viewed as an affine generalization of the Zhu algebra isomorphism for positive and integral level affine VOAs $\A(L_{\hat{\mathfrak{g}}}(k,0))\cong U(\mathfrak{g})/\langle e_{\theta}^{k+1}\rangle$  \cite{FZ92}*{Theorem 3.1.2}.

With \cref{main:C}, we can rewrite  $(\scrU/\NN^1_\L\scrU)_d$ and $ (\scrU/\NN^1_\R\scrU)_{-d}$ in simpler forms 
\begin{align*}
(\scrU/\NN^1_\L\scrU)_d&=\spn\{e(-d), f(-d), h(-d), h(-n)h(-m): m,n>0,\ m+n=d\},\\
(\scrU/\NN^1_\R\scrU)_{-d}&=\spn\{e(d), f(d), h(d), h(n)h(m): m,n>0,\ m+n=d\}.
\end{align*} where we denote the equivalence class $a+\NN^1_\L\scrU\in \scrU/\NN^1_\L\scrU$ by $a$. Consider the following element in $\mathfrak{A}_d$ of the affine VOA $L_{\widehat{\mathfrak{sl}}_2}(1,0)$: 
\begin{equation}\label{eq:strongunitaffine}
    \begin{aligned}
  \mathscr{I}_d &= \frac{1}{3}e(-d) \otimes 1 \otimes f(d) + \frac{1}{3} f(-d) \otimes 1 \otimes e(d) + \frac{1}{6} h(-d) \otimes 1 \otimes h(d) \\
  &\ +  \sum_{n,m > 0,  n + m = d} \frac{1}{\lambda_{n,m}}h(-n)h(-m) \otimes 1 \otimes h(n) h(m),
    \end{aligned}
\end{equation}
where $1$ is the identity of the Zhu algebra $\A$ and the coefficients $\lambda_{n,m}$ are given by 
\[
\lambda_{n,m} =
\begin{cases}
4d, & \text{if } d \text{ is even and } n=m=\dfrac{d}{2},\\[6pt]
2d, & \text{otherwise.}
\end{cases}\]

 Our last  theorem proves that the element $\mathscr{I}_d$ given by  \cref{eq:strongunitaffine} satisfies the strong unital property (see \cref{strongUnitThm}). The proof uses a detailed description of the relations in the enveloping algebra $\scrU(L_{\widehat{\mathfrak{sl}}_2}(1,0))$ given by \cref{main:C} to
allow a direct verification that the proposed element $\mathscr{I}_d$ from \cref{eq:strongunitaffine} acts as a two-sided strong unit (see \cref{leftUnit,rightUnit}). 
\begin{customthm}{D}[\cref{strongUnitThm}]\label{main:D}
Let    $\mathfrak{A}_d$ be the $d$-th mode transition algebra of the affine VOA $L_{\widehat{\mathfrak{sl}}_2}(1,0)$. Then $\mathscr{I}_d$ given by \cref{eq:strongunitaffine} is a strong unit of $\mathfrak{A}_d$. 
\end{customthm}

The paper is organized as follows.
Section 2 reviews preliminary definitions for vertex operator algebras, Frenkel-Zhu’s construction of the universal enveloping algebra $U(V)$, and the
filtrations/completions needed to define the mode transition algebras.
Section 3 treats the vacuum module VOAs $V_{\widehat{\mathfrak{sl}}_2}(k,0)$: we review the definition of
the vacuum module VOA, describe the corresponding enveloping algebra, and prove  \cref{noStrongUnitThm}.
Finally, Section 4 is devoted to the simple affine VOA $L_{\widehat{\mathfrak{sl}}_2}(1,0)$: we prove the
universal enveloping algebra isomorphism (see \cref{Uofsl2}) and the key relations (see \cref{relationsforUsl2}), and then use them to construct the strong units in \cref{strongUnitThm}.

\section*{Acknowledgements}
The author would like to greatly thank her research mentor Jianqi Liu for patiently introducing the problem, providing relentless support, and helping revise the paper. The author is incredibly grateful for Angela Gibney's mentorship, comments, and assistance in writing the introduction. This research was supported by the University of Pennsylvania's Center for Undergraduate Research and Fellowships (CURF).

\section{Preliminaries}

	
\subsection{Basics of VOAs} We begin by recalling the definition of vertex operator algebras (VOA) of CFT-type and Zhu's algebras, and  refer the reader to the classical texts \cites{FLM,FHL,DL,LL04,FZ92,Z,K} for more details. These notions will later be used in the construction of the mode transition algebra associated with a VOA.

\begin{definition}
A vertex operator algebra of CFT-type $(V, Y, \mathbf{1}, \omega)$ consists of a $\mathbb{N}$–graded vector space $V = \bigoplus_{n \in \mathbb{N}} V_{(n)}$ called the state equipped with a linear map (state-field correspondence) \[
        Y(\cdot, z): V \to \End(V)[[z, z^{-1}]], 
        \qquad 
        a \mapsto Y(a, z) = \sum_{n \in \mathbb{Z}} a_n z^{-n-1};
    \]  with 

\begin{enumerate}
    \item For all $n \in \mathbb{Z}$, 
    $\dim V_{(n)} < \infty$ and $\dim V_{(0)} = 1$.
    \item a distinguished vector $\mathbf{1} \in V_{(0)}$ called the vacuum vector
    \item a distinguished vector $\omega \in V_{(2)}$ called the conformal vector
    \item (operators are fields) For all $a, b \in V$, $a_n b = 0$ for sufficiently large $n >> 0$.
\end{enumerate}
satisfying the following axioms:

\begin{enumerate}

    \item[(1)] Vacuum property:
    \[
        Y(\mathbf{1}, z) = \mathrm{Id}_V.
    \]

    \item[(2)] Creation property:
    \[
        Y(a, z)\mathbf{1} = a + O(z)
    \]

    \item[(3)] Weak commutativity:  For all $a,b \in V,$ there exists  $N \in \mathbb{N}$ such that
\[
(z - w)^N \,[\, Y(a,z),\, Y(b,w) \,] = 0
\quad \text{in } \End(V)[[z, w]].
\]
    \item[(4)]  Virasoro structure: Denote $Y(\omega, z) = \sum_{n \in \mathbb{Z}} \omega_{n}z^{-n-1}$ and 
    write $L(n) = \omega_{n+1}$. There exists a constant $c_V \in \mathbb{C}$ such that  
    for all $m,n \in \mathbb{Z}$,
    \[
        [L(m), L(n)]
        = (m-n)L(m+n)
        + \frac{1}{12}(m^3 - m)\delta_{m+n,0}\, c_V .
    \] where
    \[
        L(0)a = n a \qquad \text{for all } a \in V_{(n)},\; n \in \mathbb{Z}.
    \] 
    \[
        Y(L(-1)a, z) = \frac{d}{dz} Y(a, z) \qquad \text{for all } a \in V.
    \]
\end{enumerate}
\end{definition}

\begin{lemma}[Jacobi identity]
     For $a, b \in V$ and all $m,n \in \mathbb{Z}$,
    \begin{equation}\label{eq:jacobiEq}
    (a_{(n)}b)_{(m)}= \sum_{j \ge 0} (-1)^j \binom{n}{j} 
\left( a_{(n-j)} b_{(m+j)} 
- (-1)^{n+j} \, b_{(m+n-j)} a_{(j)} \right)
\end{equation}

\end{lemma}
 We say an element $a \in V$ is homogeneous with wt$(a) = n$ if $a \in V_{(n)}$ in the grading. To define Zhu's algebra $\A(V)$ \cite{Z}, we define the two operators: For $a \in V$ homogeneous and $b \in V$,
\[
a * b = \operatorname{Res}_{z} \left( \frac{(1+z)^{\mathrm{wt}(a)}}{z} Y(a,z)b \right)
    = \sum_{i=0}^{\infty} \binom{\mathrm{wt}(a)}{i} a_{i-1} b,
\]
\[
a \circ b = \operatorname{Res}_{z} \left( \frac{(1+z)^{\mathrm{wt}(a)}}{z^{2}} Y(a,z)b \right)
    = \sum_{i=0}^{\infty} \binom{\mathrm{wt}(a)}{i} a_{i-2} b
\] and extend this operation linearly to all $V$. 

\begin{definition}
    Define $O(V)$ to be the linear span of all $a \circ b$ where $a, b \in V$ and let Zhu's algebra $\A(V) = V/O(V)$. The algebra $\A(V)$ equipped with the $*$ operation is an associative algebra  called the Zhu algebra.
\end{definition}

\subsection{Universal enveloping algebras and mode transition algebras}\label{section2.2}

By following \cite{DGK25}, we will introduce graded and split filtered completions of the universal enveloping algebra, which will be central to constructing mode transition algebras and its strong unital property.

\begin{definition} Let $V$ be a VOA and define the Borcherd's Lie algebra of
$V$ as the Lie algebra
\[
L(V) := \bigl(V \otimes \mathbb{C}((t))\bigr)\big/\Im \partial,
\]
where
\[
\partial := L(-1)\otimes 1 + 1 \otimes \frac{d}{dt}.
\]
\end{definition}
For $A \in V$ and $n \in \mathbb{Z}$ we denote by $A_{[n]}$ the image of
$A \otimes t^n$ in $L(V)$. An element $A_{[n]}$ of Borcherds' Lie algebra is meant to model the $n$-th
Fourier coefficient of the corresponding field $Y(A,z)$. Consequently, it is
natural to equip Borcherds' Lie algebra with the Lie bracket
\[
[A_{[m]}, B_{[n]}]
\;:=\;
\sum_{\ell \ge 0} \binom{m}{\ell} \bigl(A_\ell B\bigr)_{[m + n - \ell]} .
\]  

\begin{definition} For a general spanning element $a_{[n]}=a\otimes t^n+ \Im \partial \in L(V)$ with $a\in V$ homogeneous, define 
	\begin{equation}\label{eq:degofLV}
		\deg (a_{[n]}):=\wt a-n-1. 	\end{equation}
    
\end{definition}

We'll mainly adopt the notation from \cite{DGK25}*{Section 2} to present filtered completions of the universal enveloping algebra. Denote $U(L(V))$ by $\U$ for short. The degree of elements in Borcherds' Lie algebra $L(V)$  gives rise to a graded algebra structure on $\U$:
$$
		\U=\bigoplus_{d\in \Z} \U_d,\quad \U_d=\spn\left\{a^1_{[n_1]}\dots a^r_{[n_r]}\in \U: \sum\limits_{i=1}^{r}\left(\wt a^i-n_i-1\right)=d \right\}.
$$
	Let $\U_{\leq -n}=\sum_{d\leq -n} \U_d$, which makes $\U$ a split-filtered associative algebra $\U=\bigcup_{n\in \Z} \U_{\leq -n}$. Define
	\begin{equation}\label{eq:nei}
		\NN^n_\L \U=\U\cdot \U_{\leq -n}=\U\cdot L(V)_{\leq -n},\quad 	\NN^n_\sR \U=\U_{\geq n}\cdot \U=L(V)_{\geq n}\cdot \U.
	\end{equation}
	where $L(V)_{\leq -n}=\spn\{a_{[k]}\in L(V): \deg(a_{[k]})\leq -n\}$ \cite{DGK25}*{Lemma 2.4.2}. Similarly, define $L(V)_{\geq n} = \spn \{ a_{[k]} \mid \deg(a_{[k]})\geq n \}$. Since the identity $1=\vac_{[-1]}$ of $\U$ is contained in $\U_{\leq 0}$ and $\U_{\geq 0}$, we have $\NN^n_\L \U=\U=\NN^n_\sR \U	$ if $n\leq 0$. 

    The left ideals $\{\NN^n_\L \U:n\in \Z_{\geq 0}\}$ are a system of neighborhood of $0$ in $\U$, which gives a canonical seminorm on $\U$ \cite{DGK25}*{Definition A.6.1}. One can restrict these seminorms to the graded parts $\U_d$ of $\U$: 
	\begin{equation}\label{eq:completion}
		\NN^n_\L \U_d:=(\U\cdot \U_{\leq -n})_d=\sum_{j\leq -n} \U_{d-j}\cdot \U_j,\quad  	\NN^n_\sR \U_d=(\U_{\geq n}\cdot \U)_d=\sum_{i\geq n} \U_i\cdot \U_{d-i}.
	\end{equation}
	In particular, $\NN^n_\L \U_d=\NN^{n+d}_\sR \U_d$ for any $d\in \Z$. Define the completion
	$$
	\widehat{\U}_d:=\varprojlim_n\frac{\U_d}{\NN^n_\L \U_d}=\varprojlim_n\frac{\U_d}{\NN^{n+d}_\sR \U_d}\quad \mathrm{and}\quad \widehat{\U}:=\bigoplus_{d\in \Z} \widehat{\U}_d. 
	$$


	Let $J\subset \widehat{\U}$ be the graded ideal generated by the component form of Jacobi identity \eqref{eq:jacobiEq}, and let $\bar{J}\ssq \widehat{\U}$ be the closure of $J$ with respect to the seminorm defined by the image of neighborhoods \eqref{eq:nei} in $\widehat{U}$. This gives the universal enveloping algebra
	\begin{equation}\label{UEA}
    \scrU = U(V) :=\widehat{\U}/\bar{J}
    \end{equation}
	as a graded complete seminormed associative algebra with respect to the canonical seminorm induced by the image of neighborhoods \eqref{eq:nei}. The left and right neighborhoods at $0$ of $U(V)$  are given by 
	\begin{equation}\label{eq:nebors}
		\NN^n_\L \scrU=\scrU\cdot \scrU_{\leq -n}\quad \mathrm{and}\quad  	\NN^n_\sR \scrU=\scrU_{\geq n}\cdot \scrU,
	\end{equation}
	with $	\NN^{n+1}_\L \scrU_0=\sum_{j \geq n+1} \scrU_j\cdot \scrU_{-j}=	\NN^{n+1}_\sR \scrU_0$, for any $n\geq 0$.

    Note that the neighborhood $\NN^1_\L \scrU$ (resp. $\NN^1_\sR \scrU$) contains all the negatively (resp. positively) graded subspaces of $\scrU$, see \eqref{eq:completion} and \eqref{eq:nebors}. Hence the quotient modules have gradations 
	\begin{equation}\label{eq:gradationleftright}
		\scrU/	\NN^1_\L \scrU=\bigoplus_{n=0}^\infty(\scrU/	\NN^1_\L \scrU)_n,\quad \scrU/\NN^1_\sR \scrU=\bigoplus_{m=0}^\infty (\scrU/\NN^1_\sR \scrU)_{-m}.
	\end{equation}
	In particular, for any $m,n\geq 0$, we have 
	\begin{equation}\label{leftNbhdEq}
		(\scrU/	\NN^1_\L \scrU)_n =\spn \{a^1_{[n_1]}\hdots a^r_{[n_r]}+\NN^1_\L \scrU: \deg(a^i_{[n_i]})\geq 0,\ \forall i,\ \sum_{i=1}^r \deg(a^i_{[n_i]})=n \}
	\end{equation}\begin{equation}\label{rightNbhdEq}(\scrU/	\NN^1_\sR \scrU)_{-m}= \spn \{b^1_{[m_1]}\hdots b^s_{[m_s]}+\NN^1_\sR \scrU: \deg(b^j_{[m_j]})\leq 0,\ \forall j,\ \sum_{j=1}^s \deg(b^j_{[m_j]})=-m \}.
    \end{equation}
    \begin{definition}\cite{DGK25}*{Section 3.2}
        Let $V$ be a VOA. Define the underlying vector space of the mode transition algebra as
\[
\mathfrak{A} 
=  \scrU / \NN^{1}_{\L}\scrU  
\otimes_{\scrU_{0}} A \otimes_{\scrU_{0}} 
\scrU /\NN^{1}_{\R}\scrU
\]
whose gradations induce a bigrading
\[
\mathfrak{A}_{n, -m} 
= \big( \scrU / \NN^{1}_{\L}\scrU \big)_{n} 
\otimes_{\scrU_{0}} A \otimes_{\scrU_{0}} 
\big( \scrU / \NN^{1}_{\R}\scrU \big)_{-m} \quad \text{ for } n, m \in \mathbb{N}
\] where 
\[
\mathfrak{A} 
= \bigoplus_{n \in \mathbb{N}} 
   \bigoplus_{m \in \mathbb{N}} 
   \mathfrak{A}_{n, -m} . 
\] 
\end{definition} 
    The following lemma endows the vector space with an algebra structure.

    \begin{lemma}
    \cite{DGK25}*{Lemma 3.2.2}
		There is a natural linear isomorphism 
		$$\left(\scrU/	\NN^1_\L \scrU\right)\otimes_{\scrU_0} \left(\scrU/	\NN^1_\R \scrU\right)\to \A,\quad \bar{\alpha}\otimes\bar{\beta}\mapsto \alpha\circledast \beta,$$
		where $\alpha,\beta\in \scrU$ are homogeneous, and 
		\begin{equation}\label{eq:star}
			\alpha\circledast \beta=\begin{cases}
				0&\mathrm{if}\ \deg (\alpha)+\deg (\beta)\neq 0\\
				\overline{\alpha \beta}&\mathrm{if}\ \deg(\alpha)+\deg(\beta)=0,
			\end{cases}
		\end{equation}
		where $\overline{\alpha\beta}$ is the equivalence class of $\alpha\beta\in \scrU_0$ in $\A=\scrU_0/	\NN^1_\L \scrU_0$ and we extend this linearly to general products.
	\end{lemma}
    This allows us to give a well-defined associative product structure under the $\star$ operation on the mode transition algebra. Let 
\(
\alpha_i \otimes x \otimes \beta_{-j} \in \mathfrak{A}_{i,-j}
\)
and 
\(
\alpha_k \otimes y \otimes \beta_{-l} \in \mathfrak{A}_{k,-l}.
\)
We extend the the product map
\[
\star: \mathfrak{A} \times \mathfrak{A} \to \mathfrak{A}
\] linearly on homogeneous elements
\begin{equation}\label{eq:modeproduct}
(\alpha_i \otimes x \otimes \beta_{-j}) \star 
(\alpha_k \otimes y \otimes \beta_{-l})
=
\alpha_i \otimes x(\beta_{-j} \circledast \alpha_k) y 
\otimes \beta_{-l},
\end{equation}  which satisfies the compatibility condition
\[
\mathfrak{A}_{i,-j} \star \mathfrak{A}_{k,-l}
\subseteq \delta_{j,k}\,\mathfrak{A}_{i,-l}.
\] Identifying $\mathfrak{A}_{d, 0} $ with $\big( \scrU / \NN^{1}_{\L} \scrU)_d \otimes_{\scrU_0} \A$, this also induces an operation $ \mathfrak{A} \times \mathfrak{A}_{d, 0} \to \mathfrak{A}_{d, 0}$ given by 
\begin{equation}
    (\alpha_i \otimes x \otimes b_{-j}) \star (\alpha_k \otimes y) = \alpha_i \otimes x (b_{-j} \circledast \alpha_k) y
\end{equation}

\begin{definition}
    In particular, 
$\mathfrak{A}_d=\mathfrak{A}_{d,-d}$ is an associative algebra under $\star$ called the $d$-th mode transition algebra associated to $V$.
\end{definition}

\begin{definition}
    An element $\mathscr{I}_d\in \mathfrak{A}_d$ is called a strong unit if it satisfies 
		$$\mathscr{I}_d\star \mathfrak{a}=\mathfrak{a}\quad \mathrm{and}\quad \mathfrak{b}\star \mathscr{I}_d=\mathfrak{b},\quad \mathfrak{a}\in  \mathfrak{A}_{d,0},\ \mathfrak{b}\in  \mathfrak{A}_{0,-d}.$$
		$\mathfrak{A}_d$ is said to be  strongly unital if it admits a strong unit. We say that the VOA $V$ satisfies the strong unital property if its mode transition algebras $\mathfrak{A}_d$ are strongly unital for all $d\in \N$.
\end{definition}

\section{Strong unital property for the vacuum module VOA \texorpdfstring{$V_{\hatsl}(k,0)$}{}}

\subsection{Definition of vacuum module VOAs}
	First, we recall some notions in \cites{FZ92,LL04}. Let $\mathfrak{g}$ be a finite-dimensional semisimple Lie algebra with a Cartan subalgebra $\mathfrak{h}$, and let $\Delta$ be the root system associated to $\mathfrak{g}$ with root lattice $Q\subset \mathfrak{h}^\ast$. Normalize the invariant bilinear form on $\mathfrak{g}$ so that $(\theta|\theta)=2$, where $\theta$ is the longest root of $\Delta$. Let $\hat{\mathfrak{g}}=\mathfrak{g}\otimes \mathbb{C}[t,t^{-1}]\oplus \mathbb{C}K$ be its affinization with Lie bracket given by 
	$$
	[K,\hat{\mathfrak{g}}]=0,\quad [a(m),b(n)]=[a,b](m+n)+m\delta_{m+n,0} (a|b) K,\quad a,b\in \mathfrak{g},\ m,n\in \Z.
	$$
	Let $\mathbb{C} \vac$ be a $\hat{\mathfrak{g}}_{\geq 0}$-module with $K.\vac=k\vac$ and $\mathfrak{g}\otimes\mathbb{C}[t].\vac=0$. The universal vacuum module or the Weyl vacuum module 
	$$
	V_{\hat{\mathfrak{g}}}(k,0)=U(\hat{\mathfrak{g}})\otimes_{U(\hat{\mathfrak{g}}_{\geq 0})} \mathbb{C} \vac
	$$
	is a VOA equipped with the following properties:
    \begin{align*}
    Y(a(-1)\vac,z) &=\sum_{n\in \Z} a(n) z^{-n-1} , \; 
    \vac =1\otimes\vac, \; \omega_{\mathrm{aff}}=\frac{1}{2(h^\vee+k)}\sum_{i=1}^{\dim \mathfrak{g}}u^i(-1)u_i(-1)\vac.
    \end{align*} The VOA $V_{\mathfrak{\hat g}}(k,0)$ is also referred to as the vacuum module VOA of level $k$, where $h^\vee$ is the dual Coxeter number of $\Delta$, and $\{ u^i\}$ and $\{ u_i\}$ are dual orthonormal basis of $\mathfrak{g}$. For $k \neq - h^\vee$, translation operator
$T=L(-1)$ is realized as the $0$th-mode of the Sugawara $\omega_{\mathrm{aff}}$, so
\begin{equation}\label{L(-1)Operator}
L(-1)=\omega_0=\frac{1}{2(k+h^\vee)}\sum_i\sum_{m\in\mathbb Z} \normord{u^i(m)\,u_i(-m-1)}
\end{equation} where $\normord{}$ denotes normal ordering.

\subsection{Universal enveloping algebra of vacuum module VOA}\label{sec3.2}

We now turn our attention to the  universal enveloping algebra of $V = V_{\hat{\mathfrak g}}(k,0)$, the universal affine
VOA associated to a finite-dimensional semisimple Lie algebra $\mathfrak g$ at
level $k$. Following \cite{Fre07}*{Chapter 3.2} and \cite{FBZ04}*{Chapter 4} but using notation from \cites{DGK24, DGK25}, we introduce the natural
completion of its universal enveloping algebra and show it is isomorphic to $\widetilde{U}( \mathfrak{\hat{g}}, k)$.

\begin{definition}
Let $\hat{\mathfrak{g}} = \mathfrak{g} \otimes \mathbb{C}[t,t^{-1}] \oplus \mathbb{C}K$ and define the completion as the inverse limit
\[
\widetilde{U}(\hat{\mathfrak g})
:= \varprojlim_{N \ge 0} U(\hat{\mathfrak g}) / I_N,
\]
where
\begin{equation}\label{IdealN}
I_N := U(\hat{\mathfrak g}) \cdot
\bigl(\hat{\mathfrak g} \otimes t^N \mathbb{C}[[t]]\bigr)
\end{equation}
is the left ideal generated by all elements $x \otimes t^n$ with $n \ge N$. 
\end{definition}

\begin{definition}
    For a fixed $k \in \mathbb{C}$, denote $\widetilde{U}(\mathfrak{\hat g}, k)$ as the quotient of $\widetilde{U}(\mathfrak{\hat g})$ by the two-sided ideal generated by $K - k1$.
\end{definition}
A PBW basis of the vacuum module $V_{\hat{\mathfrak g}}(k,0)$ is given by
finite products of modes
\[
A
= a_{i_1}(n_1) \cdots a_{i_k}(n_k)\vac,
\qquad n_1 \le \cdots \le n_k < 0,
\]
where $\{x_i\}$ is a basis of $\mathfrak g$. The associated field
$Y(A,z)$ can be written as
\[
Y(A,z)
= \prod_{j=1}^k \frac{1}{(-n_j-1)!}\,
\normord{ \partial_z^{-n_1-1} a_{i_1}(z) \cdots
 \partial_z^{-n_k-1} a_{i_k}(z)}
\]
where $a_{i_k}(z) = Y(a_{i_k}(-1) \vac, z)$ and its $m$-th Fourier coefficient is given by
\[
A_m
= \text{Res}_{z=0} Y(A, z)z^m
\]
The normal ordering ensures that for any fixed cutoff $N \ge 0$, only finitely
many summands contain a mode $a_{i}(n)$ with $n \le N$. Thus for each mode $A_m$ we obtain a well-defined element of
$\widetilde{U}(\hat{\mathfrak g})$, and we can define a linear map
\[
f \colon L\bigl(V_{\hat{\mathfrak g}}(k,0)\bigr) \longrightarrow
\widetilde{U}(\hat{\mathfrak g}, k),\qquad
A_{[m]} \longmapsto A_m.
\] By continuity, $f$ extends uniquely to a linear map
\[
\widetilde{f} \colon
L\bigl(V_{\hat{\mathfrak g}}(k,0)\bigr)
\longrightarrow \widetilde{U}(\hat{\mathfrak g}, k).
\]

\begin{lemma}[{\cite{Fre07}*{Proposition 3.2.1}}]
The map
\[
\widetilde{f} \colon
L\bigl(V_{\hat{\mathfrak g}}(k,0)\bigr)
\longrightarrow \widetilde{U}(\hat{\mathfrak g}, k)
\]
is a Lie algebra homomorphism.
\end{lemma}

Passing to universal enveloping algebra $\U (V_{\hat{\mathfrak g}}(k,0))= U(L\bigl(V_{\hat{\mathfrak g}}(k,0)\bigr))$  and then taking its completion  with respect to $N_L^{n+1}(\U)$ for all $n \geq 0$, we extend via continuity to a homomorphism
\[
\phi:
\scrU(V_{\hat{\mathfrak g}}(k,0)\bigr)
\longrightarrow
\widetilde{U}(\hat{\mathfrak g}, k).
\]

\begin{lemma}
Let $V = V_{\hat{\mathfrak g}}(k,0)$ for $k \in \mathbb{C}$,
\begin{enumerate}\item[\emph{(1)}] {\cite{Fre07}*{Lemma 3.2.2}}
    The map
    \[
    \phi:  
        \scrU\bigl(V\bigr)
        \xrightarrow{ \cong }
        \widetilde{U}(\hat{\mathfrak g}, k)
    \]
    is an isomorphism of associative topological algebras.  
    Its inverse is given on generators by
    \[ 
    \widetilde{U}(\hat{\mathfrak g}, k) \to \scrU\bigl(V\bigr),
    \qquad
    a(n)\longmapsto (a(-1)\vac)_{[n]},
    \]
    for all $a\in\mathfrak g$ and $n\in\mathbb Z$.
    \item[\emph{(2)}] The Zhu algebra $A(V)$ is isomorphic to $U(\mathfrak g)$.
\end{enumerate}
\end{lemma}
\subsection{Proof of the first main theorem}\label{sec3.3}
The previous lemma is essential for constructing a usable form of the universal enveloping algebra of the vacuum module VOA useful for analyzing its mode transition algebra. Recall the spanning set of $(\scrU/\NN^1_\L \scrU)_d$ is of the form
\begin{equation*}
\{A^1_{[n_1]}\hdots A^r_{[n_r]}+\NN^1_\L \scrU: A^i \in V_{\hat{\mathfrak g}}(k,0), \ \deg(A^i_{[n_i]})\geq 0,\ \forall i,\ \sum_{i=1}^r \deg(A^i_{[n_i]})=d\}.
\end{equation*} For the case of the universal affine VOA, we claim each term $A^i_{[n_i]}$ can be expressed as the span of monomials consisting of the terms $(a^k(-1)\vac)_{[n_k]}$ where $\deg((a^k(-1)\vac)_{[n_k]}) = -n_k$. Using the Jacobi identity, we have that given $a \in \mathfrak{g}, B \in V_{\mathfrak{\hat g}}(k,0)$, 
        \[
        (a(n) B)_{[m]} = \sum_{j \geq 0} \binom{n}{j} (-1)^j (a(-1)\vac)_{[n-j]}B_{[m+j]} - \sum_{j \geq 0} \binom{n}{j}(-1)^{n+j}  B_{[m+n-j]} (a(-1)\vac)_{[j]}.
        \]

   
   The claim then follows by applying an inductive argument on each length $k$ monomial  \[(a(-n_1) \hdots a(-n_k)\vac)_{[m]}.\]
       Hence for $a\in \mathfrak{g}$ and $n\in \Z$, we can denote the element $(a(-1)\vac)_{[n]}\in L(V)$ by the symbol $a(n)$ where 
	$\deg a(n)=-n$. For $\scrU=\scrU(V_{\hat{\mathfrak{g}}}(k,0))$, we can rewrite \cref{leftNbhdEq,rightNbhdEq} as
	\begin{equation}\label{leftNbhd} (\scrU/\NN^1_\L \scrU)_d = \{a_1(-r_1)\hdots a_m(-r_m) : r_1\geq \hdots \geq r_m\geq 0,\ a^i\in \mathfrak{g},\  r_1+\hdots+r_m=d \}
    \end{equation}     \begin{equation} \label{rightNbhd}(\scrU/\NN^1_\R \scrU)_d = \{b_n(s_n)\dots b_1(s_1): s_1\geq \dots \geq s_n\geq 0,\ b^j\in \mathfrak{g},\  s_1+\dots+s_n=d  \} \end{equation}
	where we use the same notation $a^1(-r_1)\dots a^m(-r_m)$ for its class in $\scrU/\NN^1_\L \scrU$. 
	
	We observe the following facts for $\mathfrak{A}_1$:
	\begin{enumerate}
		\item As a vector space, 
		\begin{align*}
			\mathfrak{A}_1&=(\scrU/\NN^1_\L \scrU)_1\otimes _{\scrU_0}U(\mathfrak{g})\otimes_{\scrU_0} (\scrU/\NN^1_\R \scrU)_{-1}\\
			&=\spn\{a(-1) \otimes x \otimes b(1): a,b\in \mathfrak{g}, x\in U(\mathfrak{g}) \},
		\end{align*}
        where $U(\mathfrak{g})=\A\cong \scrU_0/\sum_{i\geq 1} \scrU_{i}\scrU_{-i}$.
		
		\item For homogeneous $a,b \in \scrU$ such that $ab\in \scrU_0$, if $\deg a<0$, then $\overline{ba}=0$ in $\A$ and 
		$$a \circledast b= \overline{ab}=\overline{ba+[a,b]}=\overline{[a,b]}.$$
		Their product is given by
		\begin{equation}\label{eq:A1prod}
			\begin{aligned}
				&(a(-1)\otimes x\otimes b(1))\star (a'(-1)\otimes x'\otimes b'(1))\\
				&=a(-1)\otimes (x b(1)\circledast a'(-1)x')\otimes b'(1)\\
				&=a(-1)\otimes x( [b,a']+(b|a')k )x'\otimes b'(1)\\
				&=a(-1)\otimes x\otimes(( [b,a']+(b|a')k )x'b'(1))
            \end{aligned}
		\end{equation}
    \end{enumerate}

    \begin{lemma}\label[lemma]{linearInd}
        For $\scrU=\scrU(V_{\hatsl}(k,0))$ and $k \in \mathbb{C}$ not equal to $-2$, consider the mode transition algebra \[\mathfrak{A}_{0,-1}=\A\otimes_{\scrU_0}(\scrU/\NN_\R^1\scrU)_{-1}.\] Then \[\{1\otimes e(1), 1\otimes f(1), 1\otimes h(1)\}\] is an $\A$-basis of $\mathfrak{A}_{0, -1}$ as a left $\A$-module where $\A = U(\mathfrak{sl}_2)$.
    \end{lemma}

    \begin{proof}
        It follows from the definition of   \cref{rightNbhd} that the set $\{1\otimes e(1), 1\otimes f(1), 1\otimes h(1)\}$ spans $\mathfrak{A}_{0, -1}$ as an $\A$-module. It remains to prove that they are $\A$-linearly independent. For $\lambda, \mu, \nu \in \A$, let \begin{equation}\label{linIndEq}   \lambda \cdot 1 \otimes h(1)+\mu \cdot 1 \otimes e(1) + \nu \cdot 1 \otimes f(1) =0 \end{equation} There is a well-defined bilinear action given by
\begin{align*}
    &  \A \otimes_{\scrU_0} (\scrU/\NN_\R^1 \scrU)_{-1} \times (\scrU/\NN_\L^1 \scrU)_{1} \to \A \otimes_{\scrU_0} \A \cong \A, \\
     & (x \otimes \alpha(1), \beta(-1)) \mapsto x \otimes(\alpha(1) \circledast \beta(-1)) = x  \otimes [\alpha(1), \beta(-1)]. 
\end{align*}
 Applying an action of $(-)\circledast h(-1)$ to \cref{linIndEq}, we have
        \begin{align*}
            (\lambda \cdot 1 \otimes h(1)+\mu \cdot 1 \otimes e(1) + \nu \cdot 1 \otimes f(1) =0)\circledast h(-1) 
            &=     2\lambda k \otimes 1 -2\mu \otimes e(0) + 2\nu \otimes f(0) = 0.
        \end{align*} Passing through the isomorphism $\scrU_0 \cong \A$, this implies that $\lambda k - \mu e +  \nu f = 0$. Applying the action $(-) \circledast  e(-1)$ and $(-) \circledast f(-1)$ to \cref{linIndEq} gives us two more equations $2 \lambda e - \nu h + \nu k = -2 \lambda f + \mu h + \mu k = 0$. We now show that $\lambda=\mu=\nu=0$. From the three actions, we obtain the system of equations
\begin{align}
0&=\lambda  k-\mu e+\nu f, \label{eq:lin1}\\
0&=2\lambda e-\nu(h-k), \label{eq:lin2}\\
0&=-2\lambda f+\mu(h+k). \label{eq:lin3}
\end{align}
We rewrite $\ref{eq:lin2}$ and $\ref{eq:lin3}$ as
\begin{equation}\label{eq:lin4}
2\lambda e=\nu(h-k), \qquad 2\lambda f=\mu(h+k).
\end{equation}

Multiply the first identity in \eqref{eq:lin4} on the right by $f$ and use $(h-k)f=f(h-k-2)$ (since $hf=f(h-2)$) to get
\begin{equation}\label{eq:lin5}
2\lambda ef=\nu f(h-k-2).
\end{equation}
Multiply the second identity in \eqref{eq:lin4} on the right by $e$ and use $(h+k)e=e(h+k+2)$ (since $he=e(h+2)$) to get
\begin{equation}\label{eq:lin6}
2\lambda fe=\mu e(h+k+2).
\end{equation}
Subtracting \eqref{eq:lin6} from \eqref{eq:lin5} and using $ef-fe=h$ yields
\begin{equation}\label{eq:lin7}
2\lambda h=\nu f(h-k-2)-\mu e(h+k+2).
\end{equation}

Now use \eqref{eq:lin1} to substitute $\mu e=k\lambda+\nu f$ into \eqref{eq:lin7}. This gives
\[
2\lambda h
=\nu f(h-k-2)-(k\lambda+\nu f)(h+k+2)
=-2(k+2)\nu f-k\lambda(h+k+2),
\]
hence
\begin{equation}\label{eq:lin8}
2(k+2)\nu f=-\lambda(k+2)(h+k).
\end{equation}
If $k\neq -2$, we may divide by the scalar $k+2$ to obtain
\begin{equation}\label{eq:lin9}
\nu f=-\tfrac12\,\lambda(h+k).
\end{equation}
Then \eqref{eq:lin1} implies
\begin{equation}\label{eq:lin10}
\mu e=k\lambda+\nu f=\tfrac12\,\lambda(k-h)=-\tfrac12\,\lambda(h-k).
\end{equation}

Substitute \eqref{eq:lin10} into \eqref{eq:lin6}:
\[
2\lambda fe=\mu e(h+k+2)=-\tfrac12\,\lambda(h-k)(h+k+2),
\]
so
\begin{equation}\label{eq:lin11}
\lambda\Bigl(4fe+(h-k)(h+k+2)\Bigr)=0.
\end{equation}
Expanding $(h-k)(h+k+2)=h^2+2h-k(k+2)$, \eqref{eq:lin11} becomes
\begin{equation}\label{eq:lin12}
\lambda\Bigl(\Omega-k(k+2)\Bigr)=0,
\qquad\text{where }\Omega:=4fe+h^2+2h\in \A
\end{equation}
is the Casimir element of $\sl$. Since $\A=U(\mathfrak{sl}_2)$ is an integral domain 
and $\Omega-k(k+2)\neq 0$, \eqref{eq:lin12} forces $\lambda=0$. Then \eqref{eq:lin4} gives $\nu(h-k)=0$ and $\mu(h+k)=0$, hence $\mu=\nu=0$.
Therefore $ \lambda \cdot 1 \otimes h(1)+\mu \cdot 1 \otimes e(1) + \nu \cdot 1 \otimes f(1)=0$ implies $\lambda=\mu=\nu=0$,
and so we've shown that $\{1\otimes h(1),1\otimes e(1),1\otimes f(1)\}$ is $\A$-linearly independent.  
    \end{proof}
    \begin{theorem}\label{noStrongUnitThm}
        The vacuum module VOA $V_{\hatsl}(k,0)$ does not satisfy the strong unital property for all non-critical level $k \in \mathbb{C}$ with $k \neq -2$.
        \end{theorem}

        \begin{proof}
            We will assume there exists a strong unit for $\mathfrak{A}_1$ and derive a contradiction. First, let $\alpha,\beta \in \{e, f, h\}$, then the unit element is of the form 
\[
\mathscr{I}_1 = \sum_{\alpha, \beta} \lambda_{\alpha, \beta} \alpha(-1) \otimes x_{\alpha, \beta} \otimes \beta(1) \in \mathfrak{A}_1.
\]  where $x_{\alpha, \beta} \in U(\sl)$. Let $1 \otimes b(1) \in \mathfrak{A}_{0, -1}$, then our simplifications from \ref{eq:A1prod} show that
\[
(1 \otimes b(1)) \star \mathscr{I}_1 = 1 \otimes \sum_{\alpha, \beta}\lambda_{\alpha, \beta} ([b, \alpha](0)+ k(b \mid \alpha) x_{\alpha, \beta}\beta(1) =1 \otimes b(1). 
\]

Let $b = e$. For the first sum, using
\([e,e]=0\), \([e,f]=h\), \([e,h]=-2e\), we obtain
\[
\begin{aligned}
&\sum_{\alpha,\beta} \lambda_{\alpha,\beta}[e, \alpha](0)x_{\alpha, \beta}\beta(1) \\
&= \sum_{\beta}\lambda_{e,\beta}[e, e](0)x_{e, \beta}\beta(1)
 + \sum_{\beta}\lambda_{f,\beta}[e, f](0)x_{f, \beta}\beta(1)
 + \sum_{\beta}\lambda_{h,\beta}[e, h](0)x_{h, \beta}\beta(1) \\
&= \sum_{\beta}\lambda_{f,\beta}h(0)x_{f, \beta}\beta(1)
   -2\sum_{\beta}\lambda_{h,\beta}e(0)x_{h, \beta}\beta(1) \\
&= \lambda_{f,e}h(0)x_{f,e}e(1) + \lambda_{f,f}h(0)x_{f,f}f(1) + \lambda_{f,h}h(0)x_{f,h}h(1) -2\lambda_{h,e}e(0)x_{h,e}e(1) \\ &-2\lambda_{h,f}e(0)x_{h,f}f(1) - 2\lambda_{h,h}e(0)x_{h,h}h(1).
\end{aligned}
\]

Recall $(e\mid f)=1$ and $(e\mid\alpha)=0$ for $\alpha \neq f$, so for the second sum, we have
\[
\sum_{\alpha,\beta} \lambda_{\alpha,\beta}k(e\mid\alpha)x_{\alpha, \beta}\beta(1)
= k(\lambda_{f,e}x_{f, e}e(1)+\lambda_{f,f}x_{f,f}f(1)+\lambda_{f,h}x_{f,h}h(1)).
\] Putting the two sums together, then the following equation must be true
\begin{align*}
&1 \otimes e(1) \\
&= \lambda_{f,e}h(0)\otimes e(1)
 + \lambda_{f,f}hx_{f,f}\otimes f(1)+\lambda_{f,h}hx_{f, h} \otimes h(1) -2\lambda_{h,e}ex_{h, e} \otimes e(1) \\ &\ \ -2\lambda_{h,f}ex_{h,f} \otimes f(1)  -2\lambda_{h,h}ex_{h,h} \otimes h(1)
  \\ &+ k\lambda_{f,f}x_{f,f}\otimes f(1)+k\lambda_{f,h}x_{f,h}\otimes h(1) +k\lambda_{f,e}x_{f,e}\otimes e(1) \\
 &= (\lambda_{f,e}hx_{f,e} -2 \lambda_{h,e} ex_{h,e} + k \lambda_{f,e}x_{f,e}) \otimes e(1)  +  (\lambda_{f,f}hx_{f,f} - 2 \lambda_{h,f} ex_{h,f} + k \lambda_{f,f}x_{f,f})\otimes f(1)  \\ &\ \ +(\lambda_{f,h} hx_{f,h} - 2 \lambda_{h,h} e x_{h,h} + k \lambda_{f,h}x_{f,h}) \otimes h(1)
\end{align*} By \cref{linearInd}, each term in the set $\{1 \otimes e(1), 1 \otimes h(1), 1 \otimes f(1)\}$ is linearly independent over $\A$, so we obtain the condition  
\begin{equation} \label{contradiction}\lambda_{f,e}hx_{f,e} + k \lambda_{f,e}x_{f,e} -2 \lambda_{h,e}e x_{h,e} 
 = \lambda_{f,e} (k+h)x_{f,e}- 2 \lambda_{h,e} e x_{h,e} = 1.\end{equation} Consider a highest weight $\sl$-module $L(k)$ with highest weight $k \in \mathbb{C}$ where $k \neq 2$ and $v_k$ is a highest weight vector.  Define an anti-involution $\theta: U(\sl) \to U(\sl)$ defined on the generators by $X \mapsto -X$ for $X\in \sl$. The anti-involution $\theta$ provides a right $\sl$-module structure on $L(k)$ with the actions $v_k e = 0$ and $v_k h= -kv_k$.  Substituting our condition from \ref{contradiction} to
 \begin{equation*}
     v_k = v_k \cdot 1= v_k \cdot (\lambda_{f,e} (k+h)x_{f,e}- 2 \lambda_{h,e} e x_{h,e})  =0
 \end{equation*}
 gives us a contradiction because the highest weight vector must be nonzero.
 
\end{proof}
\section{Strong Unital Property For The Affine Module VOA \texorpdfstring{$L_{\widehat{\mathfrak{sl}}_2}(1,0)$}{}}

\subsection{Definition of affine module VOAs}
    Let $J_k$ denote the unique proper, maximal ideal of $V_{\hat{\mathfrak g}}(k,0)$ then the quotient 
    \[
    L_{\hat{\mathfrak g}}(k,0) = V_{\hat{\mathfrak g}}(k, 0)/ J_k
    \]
    is called the simple affine VOA of level $k$. The ideal $J_k$ is also the unique largest proper $\hat{\mathfrak g}$-module of the vacuum module VOA, and $L_{\hat{\mathfrak g}}(k,0)$ is an irreducible $\hat{\mathfrak g}$-module \cite{LL04}. For the case $\mathfrak{g} = \sl$ and level $k = 1$, the simple affine VOA $L_{\hatsl}(1,0) = V_{\hatsl}(1,0) / \langle e(-1)e(-1)\vac \rangle$ where $\langle e(-1)e(-1) \vac \rangle = U(\hatsl) \cdot e(-1)e(-1)\vac$  \cites{K74,FK80}.
    \begin{lemma}\cite{FZ92}
        The Zhu algebra $A(L_{\hatsl}(1,0))$ is isomorphic to $U(\sl)/\langle e^2 \rangle$.
    \end{lemma}

\subsection{The universal enveloping algebra of the affine VOA for \texorpdfstring{$\sl$}{}}
 In this section, we prove one of our main theorems by establishing the isomorphism 
\[
\scrU\bigl(L_{\widehat{\mathfrak{sl}}_2}(1,0)\bigr)
\cong
\widetilde{U}(\widehat{\mathfrak{sl}}_2, 1)\big/ \overline{\langle e(-1)e(-1)\rangle}.
\] where $
\langle e(-1)e(-1)\rangle = U(\widehat{\mathfrak{sl}}_2, 1)\,e(-1)e(-1)\,U(\widehat{\mathfrak{sl}}_2, 1) $
is the two–sided ideal of $U(\widehat{\mathfrak{sl}}_2, 1)$ generated by $e(-1)e(-1)$. Then, we will derive a complete set of relations in this quotient and use them to
describe an explicit spanning set for the universal enveloping algebra of $L_{\hatsl}(1,0)$.

For the proofs, we will work entirely within the quotient algebra $\widetilde{U}(\widehat{\mathfrak{sl}}_2, 1)\big/\langle \overline{e(-1)e(-1)}\rangle$, meaning we will denote the equivalence class $[a(n)] \in \widetilde{U}(\widehat{\mathfrak{sl}}_2, 1)\big/\langle \overline{e(-1)e(-1)}\rangle$ simply as $a(n)$.

\begin{lemma}\label[lemma]{baseCase}
Consider the quotient algebra $\widetilde{U}(\widehat{\mathfrak{sl}}_2, 1)\big/\langle \overline{e(-1)e(-1)}\rangle
$, then \(e(k)e(-1) = 0\) for all $k \in \mathbb{Z}$.
\end{lemma}

\begin{proof}
Fix $k \in \mathbb{Z}$, then we know the Lie bracket
\begin{align*}
[h(k + 1), e(-1)e(-1)] &= [h(k + 1), e(-1)]e(-1) + e(-1)[h(k + 1), e(-1)] \\
&= 2e(k)e(-1) + 2e(-1)e(k)
\end{align*} is equal to $0$. Since the $e$ terms commute, then $2e(k)e(-1) + 2e(-1)e(k) = 4e(k)e(-1) = 0$ implies $e(k)e(-1) = 0$.
\end{proof}
    \begin{corollary}\label[lemma]{allZero}
        Consider the quotient algebra $\widetilde{U}(\widehat{\mathfrak{sl}}_2, 1)\big/\langle \overline{e(-1)e(-1)}\rangle$, then 
        $e(k)e(l) = 0$ for all $k, l \in \mathbb{Z}$
    \end{corollary}
    
    \begin{proof}
        Fix $k, l\in \mathbb{Z}$, then we know the Lie bracket 
        \begin{align*}
        [h(l+1), e(k)e(-1)] &= [h(l+1), e(k)]e(-1) + e(k)[h(l+1), e(-1)] \\
        &= 2e(k+l+1)e(-1) + 2e(k)e(l)
        \end{align*} is equal to $0$. By \cref{baseCase}, the term $2e(k+l+1)e(-1) = 0$ which leaves $e(k)e(l) = 0$.
    \end{proof} 
    \begin{theorem}\label{Uofsl2}
     There exists a continuous isomorphism between topological associative algebras
\[
\tilde{\psi}:
\scrU(L_{\widehat{\mathfrak{sl}}_2}(1,0))
\overset{\cong}\to 
\widetilde{U}(\widehat{\mathfrak{sl}}_2, 1) / \overline{\langle e(-1)e(-1) \rangle},
\]
where $\overline{\langle e(-1)e(-1)\rangle}$ is the closure of the two-sided ideal generated by $e(-1)e(-1)\in \widetilde{U}(\widehat{\mathfrak{sl}}_2, 1)$. 
    \end{theorem}
    \begin{proof}
        Let $V = V_{\hatsl}(1,0)$ and $I = \langle e(-1)e(-1)\vac \rangle$. Consider the induced image of the quotient  $\overline \partial: V/I  \otimes  \mathbb{C}((t)) \to V/I \otimes \mathbb{C}((t)), [v] \otimes t^m \mapsto [L(-1)v] \otimes t^m + [v] \otimes mt^{m-1}$. This image is well-defined because $L(-1)$ viewed in the form of  \cref{L(-1)Operator}
        preserves the ideal generated by $e(-1)e(-1)\vac$. 
        The completion of the Borcherd's Lie algebra for $L_{\hatsl}(1,0)$ can be expressed as
        \begin{align*}
            L(L_{\widehat {\sl}}(1,0)) &= \frac{L_{\widehat {\sl}}(1,0) \otimes \mathbb{C}((t))}{\Im \overline{\partial}}  = \frac{V_{\widehat {\sl}}(1,0)/\langle e(-1)e(-1) \vac \rangle \otimes \mathbb{C}((t))}{\Im \overline{\partial}}
        \end{align*}  where $\Im \overline \partial$ is viewed as the quotient space  
        \[
        \Im \overline \partial \cong \frac{ \Im \partial + \langle e(-1)e(-1)\vac \rangle \otimes \mathbb{C}((t))}{\langle e(-1)e(-1)\vac \rangle \otimes \mathbb{C}((t))}.
        \] Then we have the following canonical vector space isomorphism
        \begin{align*}
            L(L_{\widehat {\sl}}(1,0))&\cong \frac{V_{\widehat {\sl}}(1,0) \otimes \mathbb{C}((t)) / \langle e(-1)e(-1)\vac \rangle \otimes \mathbb{C}((t))}{ \Im \partial + \langle e(-1)e(-1)\vac \rangle \otimes \mathbb{C}((t))/\langle e(-1)e(-1)\vac \rangle \otimes \mathbb{C}((t))} \\
            &\cong  \frac{V_{\widehat {\sl}}(1,0) \otimes \mathbb{C}((t))}{\Im \partial + \langle e(-1)e(-1)\vac \rangle} \\
            &\cong \frac{V_{\widehat {\sl}}(1,0) \otimes \mathbb{C}((t)) / \Im \partial }{\Im \partial + \langle e(-1)e(-1)\vac \rangle / \Im \partial} \\
            &\cong \frac{L(V_{\widehat {\sl}}(1,0))}{L(\langle e(-1)e(-1)\vac \rangle)} 
        \end{align*}
         For the linear map $f: L(V_{\hat{\mathfrak{g}}}(k,0)) \to \widetilde{U}(\hat{\mathfrak{g}}, k), {A}_{[m]} \mapsto {A}_m$ to descend to the quotient map 
         \[
         g: \frac{L(V_{\widehat {\sl}}(1,0))}{L(\langle e(-1)e(-1)\vac \rangle)} \to \widetilde{U}(\hatsl, 1)/ \overline{\langle e(-1)e(-1) \rangle}, \overline{A_{[m]}} \mapsto \overline{f(A_{[m]})},
         \] we need to show $f(L(\langle e(-1)e(-1)\vac)) \subseteq \langle e(-1)e(-1)\rangle$ for all $m \geq 0$. Since the Borcherd Lie algebra of the ideal generated by $\langle e(-1)e(-1)\vac \rangle$ is spanned by elements of the form $(U(\widehat{\mathfrak{sl}}_2, 1)\cdot e(-1)e(-1)\vac)_{[m]}$ for any $m \in \mathbb{Z}$, we'll proceed by induction on the length $k$ of its monomial elements
         \[(a_{i_1}(n_1) \hdots a_{i_k}(n_k) \cdot e(-1)e(-1)\vac)_{[m]} \in (U(\widehat{\mathfrak{sl}}_2, 1)\cdot e(-1)e(-1)\vac)_{[m]}.\] For the base case $k=0$, 
\begin{align*}
    f((e(-1)e(-1)\vac)_{[m]})
    &= \Res_z\, Y(e(-1)e(-1)\vac,z)\,z^m \\
    &= \Res_z\, :e(z)e(z):\,z^m \\
    &= \Res_z \left( \sum_{n\in\Z} \Big(\sum_{k+l=n} \normord{e(k) e(l)}\Big)
        z^{-n-2} z^m \right).
\end{align*}
The residue is the coefficient of $z^{-1}$, so we require $n=m-1$, and hence
\begin{equation}
\label{CheckIdeal}
f((e(-1)e(-1)\vac)_{[m]})
    = \sum_{k+\ell=m-1} \normord{e(k) e(l)}
\end{equation}
which vanishes in the quotient by \cref{allZero}. Assume   \[f((a_{i_1}(n_1) \hdots a_{i_l}(n_l)e(-1)e(-1)\vac)_{[m]}) = \overline 0\] for $l = k$ and all $m \in \mathbb{Z}$ and we want to show this holds for $l = k+1$ and all $m$. By the Jacobi identity,
\begin{align*}
    &f \bigl( a_{i_0}(n_0)a_{i_1}(n_1) \hdots a_{i_k}(n_k)e(-1)e(-1)\vac)_{[m]}\bigr) \\
    &= \bigl(a_{i_0}(n_0)a_{i_1}(n_1) \hdots a_{i_k}(n_k)e(-1)e(-1)\vac \bigr)_{m} \\
    &= \bigl(((a_{i_0}(-1)\vac)_{n_0})a_{i_1}(n_1) \hdots a_{i_k}(n_k)e(-1)e(-1)\vac \bigr)_{m} \\
    &= \sum_{j \geq 0} \binom{n_0}{j}(-1)^j \bigl(a_{i_0}(-1)\vac\bigr)_{n_0-j}\bigl(a_{i_1}(n_1) \hdots a_{i_k}(n_k)e(-1)e(-1)\vac \bigr)_{m+j}   \\ &- \sum_{j \geq 0} (-1)^{n_0+j} \binom{n_0}{j} \bigl(a_{i_1}(n_1) \hdots a_{i_k}(n_k)e(-1)e(-1)\vac \bigr)_{m+n_0-j} \bigl(a_{i_0}(-1)\vac \bigr)_j.
\end{align*}
The inductive hypothesis implies both terms must be in $\langle e(-1)e(-1) \rangle =U(\widehat{\mathfrak{sl}}_2, 1) \cdot  e(-1)e(-1) \cdot U(\widehat{\mathfrak{sl}}_2, 1)$ which proves $g$ is well-defined. We can further extend the map \[L(L_{\hatsl}(1,0)) \xrightarrow[]{\cong} \frac{L(V_{\widehat {\sl}}(1,0))}{L(\langle e(-1)e(-1)\vac \rangle)} \xrightarrow[]{g} \widetilde{U}(\hatsl, 1)/ \overline{\langle e(-1)e(-1) \rangle} \] to a canonical map between universal enveloping algebras 
\[
\psi: U(L(V_{\hatsl}(1,0))) \xrightarrow[]{\cong} \frac{U(L(V_{\hatsl}(1,0)))}{\langle L(\langle e(-1)e(-1)\vac \rangle} \to \widetilde{U}(\hatsl, 1)/ \overline{\langle e(-1)e(-1) \rangle}
\] where $\langle L(\langle e(-1)e(-1)\vac) \rangle = U(L(V_{\hatsl}(1,0))) \cdot L(\langle e(-1)e(-1)\vac \rangle) \cdot U(L(V_{\hatsl}(1,0)))$.

Since $\psi(\NN^{n+1}_\L U(L(V_{\hatsl}(1,0))))\ssq I_{n+1}(\widetilde{U}(\hatsl, 1))$, the universal property of completions induces a continuous map between topological associative algebras
\[
\widetilde{\psi}: \scrU(L_{\hatsl}(1,0)) \to 
\widetilde{U}(\hatsl, 1)/ \overline{\langle e(-1)e(-1) \rangle}.
\] 

Consider the map defined on generators
\[\omega: \widetilde{U}(\hatsl,1)/ \overline{\langle e(-1)e(-1) \rangle} \to \scrU(L_{\hatsl}(1,0)), x(n) \mapsto (\overline{x(-1)\vac})_{[n]}
\] and extended linearly, which is surjective because it was induced from a surjective map $\widetilde{U}(\hatsl, 1) \to \scrU(V_{\hatsl}(1,0))$. Since the $n$-th Fourier coefficient of $\omega(x(n)) = (x(-1)\vac)_{[n]}$ is exactly $x(n)$, then $\widetilde{\psi}\circ \omega =  id$.
\end{proof}

\begin{theorem}\label{relationsforUsl2}
  The associative algebra $\scrU(L_{\widehat{\mathfrak{sl}}_2}(1,0))$ is spanned by the set 
\[
\operatorname{span}\bigl\{1,\  e(m),\ f(n),\ h(r),\ h(n_1)h(n_2) \,\bigm|\, m,n,r,n_1,n_2 \in \mathbb{Z}\}.
\]  where we denote $e(m), f(n), h(r), h(n_1)h(n_2)$ by its equivalence class.
Moreover, we have the following product relations among  $e(m), f(n)$ and $h(r)$: 
    \begin{equation}
    \begin{aligned} \label{eq:relations}
    & e(x)e(y)=f(x)f(y)=0, && e(x)h(y)=-e(x+y),\\
    & h(x)f(y)=-f(x+y),   && h(x)h(y)+h(x+y)=2e(r)f(x+y-r).
    \end{aligned}
    \end{equation}Under the appropriate Lie bracket relations, we also obtain
    \begin{equation}
    \begin{aligned}\label{eq:43relations}
    & h(x)e(y)=e(x+y),      && f(x)h(y)=f(x+y),\\
    & h(x)h(y)-h(x+y) = 2f(x+y-r)e(r).
    \end{aligned}
    \end{equation}
for all $x,y,r\in \Z$. It follows the subset is dense in the topological algebra $\mathscr{U}(L_{\hatsl}(1,0))$.
\end{theorem}

\begin{proof}
 Applying
$\text{ad}_{f(k)},\text{ad}_{f(\ell)},$ then $\text{ad}_{f(m)}$ successively for free variables $m, l, k \in \mathbb{Z}$ to the relation
$e(n_1)e(n_2)=0$ yields the three equations (respectively):
 \begin{equation}\label{firstRel} 
    e(n_2)h(k+n_1) + 2e(k+n_1+n_2) + e(n_1)h(k+n_2) = 0,
  \end{equation}\begin{equation}\begin{aligned}\label{secondRel} 
    0 &= -h(\ell+n_2)h(k+n_1)
    + 2e(n_2)f(\ell+k+n_1)
    - 2h(\ell+k+n_1+n_2)\\
    &- h(\ell+n_1)h(k+n_2)
    + 2e(n_1)f(\ell+k+n_2),
  \end{aligned}\end{equation}
    \begin{equation}\label{thirdRel}
    \begin{aligned} 
    0 &= -2h(k+n_1)f(m+\ell+n_2)
        -2h(\ell+n_2)f(m+k+n_1)
        -2h(k+n_2)f(m+\ell+n_1) \\
      &\quad -2h(\ell+n_1)f(m+k+n_2)
        -2h(m+n_1)f(\ell+k+n_2)
        -2(m+n_2)f(\ell+k+n_1) \\
      &\quad -8f(m+\ell+k+n_1+n_2).
    \end{aligned}\end{equation}

Specializing to \(n_1 = n_2 = r\),  \cref{firstRel} gives
\[
2e(r)h(k+r) + 2e(k+2r) = 0
\quad\Rightarrow\quad
e(r)h(k+r) = -\,e(k+2r).
\]
Setting \(n_1 = n_2 = r\) in \cref{secondRel} and writing
\(x = k+r\), \(y = \ell+r\), we obtain
\[
-2h(y)h(x) - 2h(x+y) + 4e(r)f(x+y-r) = 0 \implies
h(x)h(y) + h(x+y)= 2e(r)f(x+y-r).
\]
Also, \cref{thirdRel} simplifies to
\[
h(k+r)f(2k+r) + f(3k+2r) = 0 \implies
h(k+r)f(2k+r) = -\,f(3k+2r).
\]

If $k = -1$ and $r = 1$, we have $h(0)f(-1) + f(-1) = 0$. The bracket \[[f(-1), h(0)f(-1) + f(-1)] = 2f(-1)f(-1) = 0\] implies that $f(-1)f(-1) = 0$. Consider the involution 
\[
\alpha: U(\hatsl, 1) \to U(\hatsl, 1), e(n) \mapsto -f(n), f(n) \mapsto -e(n), h(n) \mapsto -h(n), K \mapsto K.
\] This descends to an automorphism of $U(\hatsl, 1) / \langle e(-1)e(-1) \rangle$ because \[\alpha(e(-1)e(-1)) = f(-1)f(-1) = 0.\] The map $\alpha(e(n_1)e(n_2)) = f(n_1)f(n_2) = 0$ for all $n_1, n_2 \in \mathbb{Z}$ completes our list of relations in \cref{eq:relations}. Using the Lie bracket relation of $\hatsl$, we can use \cref{eq:relations} to prove \cref{eq:43relations}.
\end{proof}

\subsection{Construction of strong units} Similar to the notation for the vacuum module VOA, we will denote $a\in \sl$, $n\in \Z$, and $a(n)$ with the equivalence class  $[a_{[n]}]\in U(L_{\hatsl}(1,0))$. From  \cref{relationsforUsl2}, we can reduce \cref{leftNbhd} and \cref{rightNbhd} to
\begin{align*}
(\scrU/\NN^1_\L\scrU)_d&=\spn\{e(-d), f(-d), h(-d), h(-n)h(-m): m,n>0,\ m+n=d\},\\
(\scrU/\NN^1_\R\scrU)_{d}&=\spn\{e(d), f(d), h(d), h(n)h(m): m,n>0,\ m+n=d\}.
\end{align*}
subject to the prescribed relations in \cref{eq:relations}. First, we introduce the following lemma to help simplify some calculations. \begin{lemma}\label[lemma]{h4}
For all $a,b,c,d\in\mathbb Z$ and for any $r,s,t\in\mathbb Z$,
\begin{align*}
h(a)h(b)h(c)h(d)
&=2e(r)f(a+b+c+d-r)-2e(s)f(a+b+c+d-s)\\ &+2e(t)f(a+b+c+d-t)-h(a+b+c+d).
\end{align*}
\end{lemma}

\begin{proof}
We repeatedly use the following relations in this proof
\[
h(x)h(y)=2e(r)f(x+y-r)-h(x+y)\quad (r\in\mathbb Z),
\qquad
f(x)h(y)=f(x+y).
\]
Starting from $h(a)h(b)h(c)h(d)$, apply the first relation to $h(a)h(b)$ with a free variable $r$:
\[
h(a)h(b)=2e(r)f(a+b-r)-h(a+b).
\]
Hence
\[
h(a)h(b)h(c)h(d)
=\bigl(2e(r)f(a+b-r)-h(a+b)\bigr)h(c)h(d).
\]
For the first term, use $f(x)h(y)=f(x+y)$ twice:
\[
2e(r)f(a+b-r)h(c)h(d)=2e(r)f(a+b+c+d-r).
\]
For the second term, reduce $h(a+b)h(c)$ by the first relation with a free variable $s$:
\[
h(a+b)h(c)=2e(s)f(a+b+c-s)-h(a+b+c).
\]
Thus
\begin{align*}
-h(a+b)h(c)h(d)
&=-2e(s)f(a+b+c-s)h(d)+h(a+b+c)h(d)
\\
&=-2e(s)f(a+b+c+d-s)+h(a+b+c)h(d).
\end{align*}
Finally, reduce $h(a+b+c)h(d)$ again with a free variable $t$:
\[
h(a+b+c)h(d)=2e(t)f(a+b+c+d-t)-h(a+b+c+d).
\]
Together, we have
\begin{align*}
h(a)h(b)h(c)h(d)
&=2e(r)f(a+b+c+d-r)-2e(s)f(a+b+c+d-s)\\ &+2e(t)f(a+b+c+d-t)
-h(a+b+c+d).
\end{align*}
\end{proof}
   \begin{theorem}\label{strongUnitThm} The simple affine VOA $V = L_{\hatsl}(1, 0)$ satisfies the strong unital property. For each $d$-th mode transition algebra $\mathfrak{A}_{d, -d}$ of $V$, the strong unit is given by
       \begin{align*}
       \mathscr{I}_d &= \frac{1}{3}e(-d) \otimes 1 \otimes f(d) + \frac{1}{3} f(-d) \otimes 1 \otimes e(d) + \frac{1}{6} h(-d) \otimes 1 \otimes h(d) \\
       &+  \sum_{n,m > 0,  n + m = d} \frac{1}{\lambda_{n,m}}h(-n)h(-m) \otimes 1 \otimes h(n) h(m) 
       \end{align*} where the coefficients are
\[
\lambda_{n,m} =
\begin{cases}
4d, & \text{if } d \text{ is even and } n=m=\dfrac{d}{2},\\[6pt]
2d, & \text{otherwise.}
\end{cases},
\]
\begin{proof}
    The proof requires checking the following two properties: 
    For all $x \otimes y \in \mathfrak{A}_{d,0}$, \[\mathscr{I}_d \star (x \otimes y) = x \otimes y\] and for all $x \otimes y \in \mathfrak{A}_{0, -d}$, \[(x \otimes y )\star  \mathscr{I}_d= x \otimes y.\] The verification of these identities  is divided into \cref{leftUnit} and \cref{rightUnit}.  One should note that the two lemmas have similar calculations due to the symmetric construction of the strong units. \cref{leftUnit} establishes that $\mathscr{I}_d$ acts as the identity on the left (in detail) and  \cref{rightUnit} establishes the action on the right. 
\end{proof}

\subsection{Proof of the strongly unital property}
\begin{lemma}\label[lemma]{leftUnit}
    For the $d$-th mode transition algebra $\mathfrak{A}_{d,-d}$ of $L_{\hatsl}(1, 0)$, 
    \[ \mathscr{I}_d \star (x \otimes y) = x \otimes y\] 
    for all spanning elements $x \otimes y \in \mathfrak{A}_{d,0}$.  
\end{lemma}

       \begin{proof}
           We check the action of the identity on each element in the spanning set for $\scrU(L_{sl_2}(1,0))$.  

           Consider $e(-d) \otimes x \in \mathfrak{A}_{d, 0}$, then 
    \begin{align*}
    &\mathscr{I}_d \star (e(-d)\otimes x) \\
    &= \frac{1}{3}e(-d)\otimes [f(d),e(-d)]\,x
     + \frac{1}{6}h(-d)\otimes [h(d),e(-d)]x \\
     &+ \sum_{n,m}\frac{1}{\lambda_{n,m}}h(-n)h(-m)\otimes [h(n)h(m),e(-d)]\,x \\
        &= \frac{1}{3}e(-d)\otimes \bigl(-h(0)+ 1\bigr)x
     + \frac{1}{6}h(-d)\otimes 2e(0)x \\
     &+ \sum_{n,m}\frac{1}{\lambda_{n,m}}h(-n)h(-m)\otimes \bigl(2e(n-d)h(m)+2h(n)e(m-d)\bigr)x \\
     &= -\frac{1}{3}e(-d)h(0) \otimes x + \frac{1}{3}e(-d) \otimes x + \frac{1}{3} h(-d)e(0) \otimes x \\
     &+  \sum_{n,m}\frac{1}{\lambda_{n,m}}h(-n)h(-m)\otimes [h(n)h(m),e(-d)]\,x
    \end{align*} The relations in \cref{eq:relations} allow us to rewrite $e(n-d)h(m) = -e(n + m-d)$ and $h(n)e(m-d) = e(n + m -d)$, which implies the summation term is equal to 0. Applying the same relations to the remaining terms allow us to reduce it further to
       \begin{align*}  
            \mathscr{I}_d \star (e(-d) \otimes x) &=  \frac{1}{3}e(-d)\otimes x+\frac{1}{3}e(-d)\otimes x+\frac{1}{3}e(-d)\otimes x \\
           &= e(-d) \otimes x
       \end{align*}
       
        Consider $f(-d) \otimes x \in \mathfrak{A}_{d, 0}$, then 
        
\begin{align*}
&\mathscr{I}_d \star (f(-d)\otimes x)\\
&= \frac{1}{3}f(-d)\otimes [e(d),f(-d)]\,x
 + \frac{1}{6}h(-d)\otimes [h(d),f(-d)]x \\
 &+ \sum_{n,m}\frac{1}{\lambda_{n,m}}\,h(-n)h(-m)\otimes [h(n)h(m),f(-d)]\,x \\
&= \frac{1}{3}f(-d)\otimes \bigl(h(0)+1\bigr)x
 - \frac{1}{3}h(-d)\otimes f(0)x \\
 &- \sum_{n,m}\frac{1}{\lambda_{n,m}}\,h(-n)h(-m)\otimes
\bigl(2f(n-d)h(m)+2h(n)f(m-d)\bigr)x \\
&= \frac{1}{3}f(-d) h(0) \otimes x + \frac{1}{3}f(-d) \otimes x - \frac{1}{3}h(-d)f(0) \otimes x \\ 
&- \sum_{n,m}\frac{1}{\lambda_{n,m}}\,h(-n)h(-m)\otimes
\bigl(2f(n-d)h(m)+2h(n)f(m-d)\bigr)x.
\end{align*}
 The relations in \cref{eq:relations} allow us to rewrite $f(n-d)h(m) = f(n+m -d)$ and $h(n)f(m-d) = -f(n+m-d)$ which implies the summation is equal to zero. We also have
        \begin{align*}
            \mathscr{I}_d \star (f(-d) \otimes x)
            &= \frac{1}{3}f(-d) \otimes x + \frac{1}{3} f(-d) \otimes x + \frac{1}{3} f(-d) \otimes x \\
            &= f(-d) \otimes x.
        \end{align*} 
        
        Consider $h(-d) \otimes x \in \mathfrak{A}_{d, 0}$, then 
        \begin{align*}
             &\mathscr{I}_d \star (h(-d) \otimes x) \\
             &= \frac{1}{3}e(-d) \otimes [f(d), h(-d)]x + \frac{1}{3}f(-d) \otimes [e(d), h(-d)]x  + \frac{1}{6} h(-d) \otimes 2x  \\
            &= \frac{2}{3} e(-d) f(0) \otimes x - \frac{2}{3}f(-d)e(0) \otimes x + \frac{1}{3}h(-d) \otimes x. 
        \end{align*} 
        From \cref{eq:relations}, then $2e(-d)f(0) = h(x)h(y) + h(-d)$ and $2f(-d)e(0) = h(x)h(y) - h(-d)$ where $x + y = -d$. This allows us to reduce the expression to  
        \begin{align*}
            &\mathscr{I}_d \star (h(-d) \otimes x) \\
            &= \frac{1}{3}(h(x)h(y) + h(-d)) \otimes x - \frac{1}{3}(h(x)h(y) - h(-d))\otimes x   + \frac{1}{3} h(-d) \otimes x\\
            &= \frac{2}{3} h(-d) \otimes x + \frac{1}{3} h(-d) \otimes x \\
            &= h(-d) \otimes x
        \end{align*}

        Consider $h(-r)h(-s) \otimes x \in \mathfrak{A}_{d,0}$ where $r,s$ are positive integers such that $r + s = d$. We only need to consider the case where $r,s > 0$ because if $s = 0$, then $h(-r)h(0) \otimes x = h(-r) \otimes h(0)x$ which reduces to the previous case. 
        \begin{align*}
             &\mathscr{I}_d \star (h(-r)h(-s) \otimes x) \\
             &= \frac{1}{3} e(-d) \otimes [f(d), h(-r)h(-s)]x + \frac{1}{3} f(-d) \otimes [e(d), h(-r)h(-s)]x \\ &+  \sum \frac{1}{\lambda_{n,m}}h(-n)h(-m) \otimes [h(n)h(m), h(-r)h(-s)]x \\
            & = \frac{2}{3}e(-d) \otimes (f(d-r)h(-s) + h(-r)f(d-s))x  \\
            &-\frac{2}{3} f(-d) \otimes (e(d-r)h(-s) + h(-r)e(d-s))x \\
            &+  \sum \frac{1}{\lambda_{n,m}}h(-n)h(-m) \otimes [h(n)h(m), h(-r)h(-s)]x.
        \end{align*} 
        Since $2f(d-r)h(-s) = 2f(d - r -s)$ and $2h(-r)f(d-s) = -2f(d - r -s)$ and similarly with the $e(\cdot)h(\cdot)$ terms, then those two terms are equal to zero and we can simply the expression to 
        \begin{align*}
            \mathscr{I}_d \star (h(-r)h(-s) \otimes x) &=  \sum \frac{1}{\lambda_{n,m}}h(-n)h(-m) \otimes [h(n)h(m), h(-r)h(-s)]x.
        \end{align*} It is straightforward to verify the following Lie bracket:
        \begin{align*}
        [h(n)h(m),\,h(-r)h(-s)]
        &= 2(m\delta_{m,r}h(n)h(-s)
        + m\delta_{m,s}h(n)h(-r) \\
        &+ n\delta_{n,r}h(-s)h(m)
        + n\delta_{n,s}h(-r)h(m)).
        \end{align*} 
        
        Assume $n \geq m$ and $r \geq s$. If $m = r$ then $n \geq m \geq s$. Since $n + m = r + s$, this implies  $n = s$. It follows that $n = m = s = r$. The same argument would apply for the case when $n = s$. If we have cases $m = s$ or $n = r$, then either case would imply that $n = r$ and $m = s$.  There are two cases in which the Lie bracket is nonzero: (1) $n = m = r =s$ (2) $n =r$ and $m = s$ and $n \neq m$. Consider the first case where $n = m = r = s$, then 
        \begin{align*}
            \mathscr{I}_d \star (h(-r)h(-r) \otimes x) &=   \frac{1}{\lambda_{r,r}}8r\cdot  h(-r)h(-r)h(r)h(-r) \otimes x.
        \end{align*} Applying  \cref{h4} by letting all the free variables be equal to $0$, then
        \begin{align*}
        h(-r)^2h(r)h(-r) &= 2e(0)f(-2r) - 2e(0)f(-2r) + 2e(0)f(-2r) - h(-2r) \\
        &= 2e(0)f(-2r) - h(-2r) \\
        &= h(-r)h(-r)
        \end{align*} and so 
        \begin{align*}
        \mathscr{I}_d \star (h(-r)h(-r) \otimes x) &= \frac8{\lambda_{r,r}} r h(-r)h(-r) \otimes x.
        \end{align*}
        
        Consider the second case where $n = r$ and $m = s$, then the Lie bracket relation simplifies to 
        \[
        [h(n)h(m), h(-r)h(-s)] = [h(r)h(s), h(-r)h(-s)] = 2(r h(-s)h(s) + sh(r)h(-r))
        \] and we can write 
        \begin{align*}
            &\mathscr{I}_d \star (h(-r)h(-s) \otimes x) \\
            &=  \frac{2}{\lambda_{r,s}}h(-r)h(-s) \otimes(r h(-s)h(s) + sh(r)h(-r))x \\
            &= \frac{2r}{\lambda_{r,s}} h(-r)h(-s)h(-s)h(s) \otimes x +  \frac{2s}{\lambda_{r,s}} h(-r)h(-s)h(r)h(-r) \otimes x \\
            &= \frac{2r}{\lambda_{r,s}} h(-s)h(-s)h(-r)h(s) \otimes x +  \frac{2s}{\lambda_{r,s}} h(-r)h(-s)h(r)h(-r) \otimes x.
        \end{align*} From \cref{h4}, we obtain
\[
h(-s)^2h(-r)h(s)
= 2e(0)f(-(s+r)) - h(-(s+r)) = h(-r)^2h(-s)h(r).
\] By the relations in \cref{eq:relations}, we know $e(0)f(-(r+s)) = h(-r)h(-s) + h(-(r+s))$ and so 
\[
h(-s)^2h(-r)h(s)
= h(-r)^2h(-s)h(r) = h(-r)h(-s).
\] Our expression in the second case is 
\begin{align*}
    \mathscr{I}_d \star (h(-r)h(-s) \otimes x) &= \frac{2(r+s)}{\lambda_{r,s}} h(-r)h(-s) \otimes x.
\end{align*} 

Finally, we verify the following: if $r = s = \frac{d}{2}$, 
\[
 \mathscr{I}_d \star (h(-r)h(-r) \otimes x) = \frac{8r}{\lambda_{r,r}} h(-r)h(-r) \otimes x = \frac{8}{4d} \frac{d}{2} h(-r)h(-r) \otimes x = h(-r)h(-r) \otimes x
\] and if $r \neq s$,
\[
 \mathscr{I}_d \star (h(-r)h(-s) \otimes x) = \frac{2(r+s)}{\lambda_{r,s}}h(-r)h(-s) \otimes x = \frac{2d}{2d} h(-r)h(-s) \otimes x = h(-r)h(-s) \otimes x.
\]
\end{proof}
       
   \end{theorem}

\begin{lemma}\label[lemma]{rightUnit}
    For the $d$-th mode transition algebra $\mathfrak{A}_{d,-d}$ of $L_{\hatsl}(1, 0)$, \[(x \otimes y) \star  \mathscr{I}_d = x \otimes y\] for all spanning elements $x \otimes y \in \mathfrak{A}_{0, -d}$. 
\end{lemma}

\begin{proof}
    Consider the element $x \otimes e(d) \in \mathfrak{A}_{0, -d}$, then 
    \begin{align*}
        &(x \otimes e(d)) \star  \mathscr{I}_d \\
        &=  \frac{1}{3} x[e(d), f(-d)] \otimes e(d) + \frac{1}{6} x[e(d), h(-d)] \otimes h(d)  \\
        &+ \sum \frac{1}{\lambda_{n,m}} x [e(d), h(-n)h(-m)] \otimes h(n)h(m) \\
        &= \frac{1}{3} x \otimes (h(0)+1)e(d) - \frac{1}{3} x \otimes e(0)h(d)\\
        &+ \sum \frac{1}{\lambda_{n,m}}x(-2e(d-n)h(-m) - 2h(-n)e(d-m)) \otimes h(n)h(m) \\
        &= \frac{1}{3}x \otimes e(d) + \frac{1}{3} x \otimes e(d) + \frac{1}{3} x \otimes e(d) \\
        &+ \sum \frac{1}{\lambda_{n,m}} x(2(d - n - m) - 2e(d -n -m)) \otimes h(n)h(m) \\
        &= x \otimes e(d).
    \end{align*}
    \noindent
    Consider $x \otimes f(d) \in \mathfrak{A}_{0, -d}$, then
    \begin{align*}
        &(x \otimes f(d)) \star  \mathscr{I}_d \\
        &= \frac{1}{3}x[f(d), e(-d)] \otimes f(d) + \frac{1}{6} x[f(d), h(-d)] \otimes h(d) \\
        &+ \sum \frac{1}{\lambda_{n,m}}x[f(d), h(-n) h(-m)] \otimes h(n)h(m) \\
        &= -\frac{1}{3} x \otimes (h(0)+1)f(d) + \frac{1}{3} x \otimes f(0)h(d)\\ &+ \sum \frac{1}{\lambda_{n,m}}x (2f(d-n)h(-m) + 2h(-n)f(d-m)) \otimes h(n)h(m)\\
        &= \frac{1}{3}x \otimes f(d) + \frac{1}{3}x \otimes f(d) + \frac{1}{3} x \otimes f(d) \\
        &= x \otimes f(d).
    \end{align*}

    \noindent Consider $x \otimes h(d) \in \mathfrak{A}_{0, -d}$, then
    \begin{align*}
        &(x \otimes h(d)) \star  \mathscr{I}_d \\
        &= \frac{1}{3}x[h(d), e(-d)] \otimes f(d) + \frac{1}{3}x[h(d), f(-d)] \otimes e(d) + \frac{1}{6}x[h(d), h(-d)] \otimes h(d) \\
        &= \frac{2}{3}x\otimes e(0)f(d) - \frac{2}{3}x \otimes f(0) e(d) + \frac{1}{3} x \otimes h(d) \\
        &= \frac{1}{3}x \otimes h(-d) + \frac{1}{3}x \otimes h(-d) + \frac{1}{3}x\otimes h(-d) \\
        &= x \otimes h(-d).
    \end{align*}
    
    \noindent Consider $x \otimes h(r)h(s)$, then 
    \begin{align*}
        &(x \otimes h(r)h(s)) \star \mathscr{I}_d  \\
        &= \frac{1}{3} x[h(r)h(s), e(-d)] \otimes f(d) + \frac{1}{3} x[h(r)h(s), f(-d)] \otimes e(d) \\
        &+ \sum \frac{1}{\lambda_{n,m}}x[h(-r)h(-s), h(-n)h(-m)] \otimes h(n)h(m) \\
        &= \frac{2}{3}x(e(r -d)h(s) + h(r) e(s-d)) \otimes f(d) - \frac{2}{3}x(f(r-d)h(s) + h(r) f(s-d)) \otimes e(d) \\
        &+ \sum \frac{1}{\lambda_{n,m}}x[h(r)h(s), h(-n)h(-m)] \otimes h(n)h(m).
    \end{align*} The first two terms cancel by the relations in \cref{eq:relations}, leaving only the summation term. It is straight-forward to verify  
    \begin{align*}
[h(r)h(s),h(-n)h(-m)]
&= 2(s\delta_{s,n}h(r)h(-m)
+ s\delta_{s,m}h(r)h(-n)
\\ &+ r\delta_{r,n}h(-m)h(s)
+ r\delta_{r,m}h(-n)h(s)).
\end{align*} A similar argument as the previous lemma allows us to consider only two cases: $(1)$ $r = s = d/2$ and $(2)$ $r \neq s$. If $r = s = d/2$, then

\begin{align*}
    (x \otimes h(r)h(r)) \star \mathscr{I}_d &= \frac{1}{\lambda_{r,r}} x(8 r h(r)h(-r)) \otimes h(r)h(r) \\
    &= \frac{8r}{\lambda_{r,r}} x \otimes h(r)h(-r)h(r)h(r) .
\end{align*} By  \cref{h4}, this is equal to 
\[
(x \otimes h(r)h(r))  \star \mathscr{I}_d = \frac{8r}{\lambda_{r,r}} x \otimes h(r)h(r) = \frac{8r}{4d} x \otimes h(r)h(r) = x \otimes h(r)h(r).
\] If $r \neq s$, then 
\begin{align*}
    (x \otimes h(r)h(s)) \star \mathscr{I}_d &= \frac{2}{\lambda_{r,s}} x (sh(-r)h(r) + rh(s)h(-s)) \otimes h(r)h(s) \\
    &= \frac{2s}{\lambda_{r,s}} x \otimes h(-r)h(r)h(r)h(s) + \frac{2r}{\lambda_{r,s}} x  \otimes h(s)h(-s)h(r)h(s) \\
    &= \frac{2(r+s)}{\lambda_{r,s}} x \otimes h(r)h(s) \\
    &= \frac{2d}{2d}x \otimes h(r)h(s) \\
    &= x \otimes h(r)h(s).
\end{align*}
\end{proof}

\end{document}